\documentclass[leqno,12pt,a4paper]{article}
\usepackage[dvips]{graphicx}
\usepackage{amssymb}
\usepackage{amsmath}
\usepackage{mathrsfs}
\usepackage{color}
\usepackage{comment}
\usepackage{theorem}
\usepackage{ascmac}
\newcommand{\pref}[1]{(\ref{#1})}
\newcommand{\be}{\begin{equation}}
\newcommand{\ee}{\end{equation}}

\newcommand{\qed}{{\unskip\nobreak\hfil\penalty50\quad\null\nobreak\hfil
	$\square$\parfillskip0pt\finalhyphendemerits0\par\medskip}}

\newcommand{\vecb}{\mbox{\boldmath $ b $}}
\newcommand{\vecc}{\mbox{\boldmath $ c $}}

\newcommand{\vecf}{\mbox{\boldmath $ f $}}

\newcommand{\vecx}{\mbox{\boldmath $ x $}}

\newcommand{\veckappa}{\mbox{\boldmath $ \kappa $}}
\newcommand{\vecnu}{\mbox{\boldmath $ \nu $}}

\newtheorem{thm}{Theorem}[section]

\theorembodyfont{\rmfamily}

\newtheorem{proof}{\normalfont\itshape Proof.}

\title{An application of interpolation inequalities between the deviation of curvature 
and the isoperimetric ratio to the length-preserving flow}
\author{Kohei Nakamura\\
Saitama University, Japan}
\date{\today}
\begin{document}
\maketitle
\begin{abstract}
In recent work of Nagasawa and the author, new interpolation inequalities 
between the deviation of curvature and the isoperimetric ratio were proved.
In this paper,
we apply such estimates to investigate the large-time behavior of the length-preserving flow of closed plane curves 
without a convexity assumption.
\end{abstract}
\section{Introduction}
\par
A number of papers have been devoted to the study of curvature flows with non-local term. 
These are called {\it non-local curvature flows}. 
Gage \cite{G} and Jiang-Pan \cite{JP} studied such flows
\begin{align}
	&\partial_t \vecf =  \veckappa - \frac{1}{L}\left(\int_0^L \veckappa \cdot \vecnu \, ds\right)\vecnu,\label{area-preserving}\\
        &\partial_t \vecf =  \veckappa - \frac{L}{2A} \vecnu,\label{JiangPan}	
\end{align}
respectively. Here $ \vecf = ( f_1 , f_2 ) \, : \, \mathbb{R} / L \mathbb{Z} \to \mathbb{R}^2 $ is a function 
such that $ \mathrm{Im} \vecf $ is a closed plane curve with rotation number $ 1 $ 
and $s$ is the arc-length parameter, $ \vecnu = ( - f_2^\prime , f_1^\prime ) $ is the inward unit normal vector,
$ \veckappa = \vecf^{ \prime \prime } $ is the curvature vector.
The (signed) area $ A $ is given by
\[
	A = - \frac 12 \int_0^L \vecf \cdot \vecnu \, ds .
\]
In \cite{G} and \cite{JP}, it was proved that a simple closed strictly convex initial curve remains so along the flow, and
the evolving curve converges to a circle in each non-local curvature flow.
However there are very few results for non-local curvature flows
when initial curve is not convex. 
Hence we would like to know the behavior
of evolving curves not assuming convexity. To do this, we consider as follows.
\par
The curvature $ \kappa = \veckappa \cdot \vecnu $ is positive when $ \mathrm{Im} \vecf $ is convex.
Since the curve has rotation number $ 1 $, the deviation of curvature is
\[
	\tilde \kappa
	=
	\kappa - \frac 1L \int_0^L \kappa \, ds
	=
	\kappa - \frac { 2 \pi } L .
\]
For a non-negative integer $ \ell $,
we set
\[
	I_\ell = L^{ 2 \ell + 1 } \int_0^L | \tilde \kappa^{( \ell )} |^2 ds ,
\]
which is a scale invariant quantity (cf.\
\cite{DKS}).
It is important to estimate $I_\ell$ for the global analysis of evolving curves.
We have the Gagliardo-Nirenberg inequalities 
\[
  I_\ell\leq CI_m^{\frac{\ell}{m}}I_0^{1-\frac{\ell}{m}},
\]
where $0\leqq \ell\leqq m$ and $C$ is constant and independent of $L$.
Such inequalities are very useful but only these are not sufficient to estimate $I_0$ because these inequalities use $I_0$.
Hence we need a different type of inequality to estimate $I_\ell$ for $\ell \geqq 0$.
\par
The curve along the flow {\rm \pref{area-preserving}} or {\rm \pref{JiangPan}} is expected to converge
to a circle when the initial curve is close to a circle (in some sense) even if it is not convex.
If it is true, the isoperimetric ratio $ \displaystyle{ \frac { 4 \pi A } { L^2 } } $ 
converges to $1$ as $t \to \infty$. Taking this into consideration, 
we introduce the quantity
\[
	I_{-1} = 1 - \frac { 4 \pi A } { L^2 }
	,
\]
which is also scale invariant,
and is non-negative by the isoperimetric inequality.  
\par
Several inequalities for $I_0, I_{-1}$ were derived by Nagasawa and the author in \cite{NN}.
Using these inequalities, $I_\ell$ can be interpolated by $I_{-1}$ and $I_m$ for $\ell \in \{0,1,\ldots,m\}$. 
The authors applied the inequalities to the flow {\rm \pref{area-preserving}} and {\rm \pref{JiangPan}}, 
and they showed that, assuming global existence, solutions of each flow become convex in finite time and converge exponentially to a circle 
even if the initial curve is not strictly convex.
\par
The purpose of this paper is to consider the large-time behavior of the length-preserving flow
\be
        \partial_t \vecf = \veckappa - \left(\frac{1}{2\pi} \int_0^L \|\veckappa\|^2 ds\right) \vecnu.
	\label{length-preserving}
\ee
This was firstly studied by Ma-Zhu \cite{MZ}, who proved that a simple closed strictly convex initial curve remains so along the flow, and
the evolving curve converges to a circle.
Their method is not applicable without the convexity. 
Hence we investigate the large-time behavior
of evolving curves not assuming the convexity.    
\par 
In section 2, we introduce the several inequalities which was proved in \cite{NN}. Furthermore
we present the results for the flow {\rm \pref{area-preserving}} and {\rm \pref{JiangPan}} 
given in \cite{NN}. 
\par
In section 3, we investigate the length-preserving flow {\rm \pref{length-preserving}}.
\section{Known results}
\setcounter{equation}{0}
\par
In this section we introduce some results which were established in \cite{NN}, for details of the proofs, see \cite{NN}.
\subsection{Several inequalities}
Clearly, we have that 
$ \tilde \kappa \equiv 0 $ implies $ \mathrm{Im} \vecf $ is a round circle,
which attains the minimum $ I_{-1} = 0 $.
This suggests that $ I_{-1} $ can be dominated by certain quantities involving $ \tilde \kappa $.
Indeed,
we have
\begin{align*}
	I_{-1}
	= & \
	\frac { L^2 - 4 \pi A } { L^2 }
	=
	\frac 1 { L^2 } \int_0^L \left( - L \vecf \cdot \veckappa + 2 \pi \vecf \cdot \vecnu \right) ds
	\\
	= & \
	- \frac 1L \int_0^L \tilde \kappa ( \vecf \cdot \vecnu ) ds
	\\
	= & \
	- \frac 1L \int_0^L \tilde \kappa \left( \vecf \cdot \vecnu - \frac 1L \int_0^L \vecf \cdot \vecnu \, ds \right) ds
\end{align*}
and
\[
	\left| \vecf \cdot \vecnu - \frac 1L \int_0^L \vecf \cdot \vecnu \, ds \right|
	\leqq L .
\]
Thus it holds that
\[
	0 \leqq
	I_{-1}
	\leqq
	I_0^{ \frac 12 }
	.
\]
However,
since $ \displaystyle{ \vecf \cdot \vecnu - \frac 1L \int_0^L \vecf \cdot \vecnu \, ds = 0 } $ when $ \tilde \kappa \equiv 0 $,
it seems that the above inequality can be improved.
\begin{thm}
We have
\[
	8\pi^2 I_{-1} \leqq
	\frac { I_0 } { 8 \pi^2 } 
        \leqq I_{-1}^{ \frac 12 }
	\left[ L^3 \int_0^L
	\left\{ \kappa^3 \tilde \kappa + ( \tilde \kappa^\prime )^2 \right\} ds
	\right]^{ \frac 12 }
	.
\]
Two equalities only hold in the trivial case $ \tilde \kappa \equiv 0 $.
\label{Theorem1}
\end{thm}
From this inequality, we have the new interpolation inequalities.
\begin{thm}
Let $ 0 \leqq \ell \leqq m $.
There exists a positive constant $ C = C ( \ell,m ) $ independent of $ L $ such that
\[
	I_\ell
	\leqq
	C \left( I_{-1}^{ \frac { m - \ell } 2 } I_m + I_{-1}^{ \frac { m - \ell } { m+1 } } I_m^{ \frac { \ell +1 } { m+1 } } \right)
\]
holds.
\label{Theorem3}
\end{thm}
\subsection{Applications to some geometric flows}
\par
In this subsection we give applications of our inequalities to the asymptotic analysis of geometric flows of closed plane curves.
One of the flows is a curvature flow {\rm \pref{JiangPan}} with a non-local term first studied by Jiang-Pan \cite{JP},
and another is the area-preserving curvature flow {\rm \pref{area-preserving}} considered by Gage \cite{G}.
If the initial curve is convex,
then the flows exist for all time, preserving the convexity,
and the curve approaches a round circle;
this was shown in \cite{JP,G}.
The local existence of flows without a convexity assumption was shown by \v{S}ev\v{c}ovi\v{c}-Yazaki \cite{SY}.
However,
the large-time behavior for this case is still open.
It seems 
finite-time blow-up may occur for some non-convex initial curves \cite{M},
but, on the other hand, the global existence for a certain initial non-convex curve was shown in \cite{SY}.
Escher-Simonett \cite{ES} showed the global existence and investigated the large-time behavior of the area-preserving curvature flow for initial data close to a circle and without a convexity assumption.
In this subsection,
we present the results for the large-time behavior of the flows without a convexity assumption {\it assuming} the global existence.
\par
We consider the flows {\rm \pref{area-preserving}} and {\rm \pref{JiangPan}} .
Observe that the equations which $ \vecf $ satisfies are
\begin{align*}
        &\partial_t \vecf = \partial_s^2 \vecf - \frac{2\pi}{L} R \partial_s \vecf,\\
	&\partial_t \vecf = \partial_s^2 \vecf - \frac{L}{2A} R \partial_s \vecf, 
\end{align*}
where
\[
	R = \left(
	\begin{array}{rr}
	0 & -1 \\
	1 & 0
	\end{array}
	\right) .
\]
Since these are parabolic equations with a non-local term,
$ \vecf $ is smooth for $ t > 0 $ as long as the solution exists.
Hence by shifting the initial time,
the initial data is smooth. Then we have the following theorem. 
\begin{thm}
Assume that $ \vecf $ is a global solution of {\rm \pref{area-preserving}} or {\rm \pref{JiangPan}} 
such that the initial rotation number is $ 1 $ and the initial {\rm (}signed{\rm )} area is positive.
Then for each $ \ell \in \mathbb{N} \cup \{ -1, 0 \} $,
there exist $ C_\ell > 0 $ and $ \lambda_\ell > 0 $ such that
\[
	I_\ell (t) \leqq C_\ell e^{ - \lambda_\ell t } .
\]
Furthermore there exist $A_\infty$ and $L_\infty$ such that 
$A(t)$ converges to $A_\infty$ and $L(t)$ converges to $L_\infty$ as $ t \to \infty$.
\label{Theorem4}
\end{thm}
\begin{thm}
Let $ \vecf $ be as in Theorem \ref{Theorem4},
and let $ \displaystyle{ f(s,t) = \sum_{ k \in \mathbb{Z} } \hat f(k) (t) \varphi_k (s) } $ be the Fourier expansion for any fixed $ t > 0 $.
Set
\[
	\vecc (t) = \frac 1 { \sqrt { L(t) } }( \Re \hat f(0) (t) , \Im \hat f(0) (t) ) ,
\]
and define $ r(t) \geqq 0 $ and $ \sigma (t) \in  \mathbb{R} / 2 \pi \mathbb{Z}  $ by
\[	
	\hat f(1) (t) = \sqrt{ L(t) } r(t) \exp \left( i \frac { 2 \pi \sigma (t) } { L(t) } \right) .
\]
Furthermore we set
\[
	\tilde { \vecf } ( \theta , t )
	=
	\vecf ( L(t) \theta - \sigma (t) , t ) ,
	\quad \mbox{for} \quad
	( \theta , t ) \in \mathbb{R} / \mathbb{Z} \times [ 0 , \infty )
	.
\]
Then the following claims hold.
\begin{itemize}
\item[{\rm (1)}]
There exists $ \vecc_\infty \in \mathbb{R}^2 $ such that
\[
	\| \vecc (t) - \vecc_\infty \| \leqq C e^{ - \gamma t } .
\]
\item[{\rm (2)}]
The function $ r(t) $ converges exponentially to the constant $ \displaystyle{ \frac { L_\infty } { 2 \pi } } $ as $ t \to \infty $:
\[
	\left| r(t) - \frac { L_\infty } { 2 \pi } \right|
	\leqq
	C e^{ - \gamma t } .
\]
\item[{\rm (3)}]
There exists $ \sigma_\infty \in \mathbb{R} / 2 \pi \mathbb{Z} $ such that
\[
	| \sigma (t) - \sigma_\infty | \leqq C e^{ - \gamma t } .
\]
\item[{\rm (4)}]
For any $ k \in \mathbb{N} \cup \{ 0 \} $ there exist $ C_k > 0 $ and $ \gamma_k > 0 $ such that 
\[
	\| \tilde { \vecf } ( \cdot , t ) - \tilde { \vecf }_\infty \|_{ C^k ( \mathbb{R} / \mathbb{Z} ) }
	\leqq
	C_k e^{ - \gamma_k t }
	,
\]
where
\[
	\tilde { \vecf }_\infty ( \theta )
	=
	\vecc_\infty + \frac { L_\infty } { 2 \pi } ( \cos 2 \pi \theta , \sin 2 \pi \theta ) .
\]
\item[{\rm (5)}]
For sufficiently large $ t $,
$ \mathrm{Im} \tilde { \vecf }( \cdot , t ) $ is the boundary of a bounded domain $ \Omega (t) $.
Furthermore, there exists $ T_\ast \geqq 0 $ such that $ \Omega (t) $ is strictly convex for $ t \geqq T_\ast $.
\item[{\rm (6)}]
Let $ D_{ r_\infty } ( \vecc_\infty ) $ be the closed disk with center $ \vecc_\infty $ and radius $ r_\infty $.
Then we have 
\[
	d_H ( \overline { \Omega(t) } , D_{ r_\infty } ( \vecc_\infty ) )
	\leqq
	C e^{ - \gamma t } ,
\]
where $ d_H $ is the Hausdorff distance.
\item[{\rm (7)}]
Let $\displaystyle 
	\vecb (t) =
	\frac 1 { A(t) } \iint_{ \Omega (t) } \vecx \, d \vecx$
be the barycenter of $\Omega(t)$. Then we have
\[
\|A(t) ( \vecb (t) - \vecc (t) )\| \leqq C e^{ - \gamma t }.
\]
\end{itemize}
\label{Theorem6}
\end{thm}
\section{The length-preserving flow}
We consider the flow {\rm \pref{length-preserving}}.	
If the initial curve is convex,
then the flow exists for all time keeping the convexity,
and the curve approaches a round circle;
this was shown in \cite{MZ}.
We have the same results for the flows {\rm \pref{area-preserving}} and {\rm \pref{JiangPan}} without convexity assumption 
{\it assuming} the global existence. In this subsection, we give a proof of this fact. 
\par
First we prove the exponential decay of $I_{-1}$.
\begin{thm}
Assume that $ \vecf $ is a global solution of {\rm \pref{length-preserving}} such that 
the initial rotation number is $ 1 $ and the initial {\rm (}signed{\rm )} area is positive.
Then there exist $C > 0$ and $\lambda > 0$ such that
\[
I_{-1}(t) 
\leqq
C e^{-\lambda t}.
\]
\label{Theorem10}
\end{thm}
\begin{proof}
Since
\be
\frac{dA}{dt}
=
-\int_0^L \partial_t \vecf \cdot \vecnu ds
=
\frac{L}{2\pi} \int_0^L \tilde{\kappa}^2 ds
=
\frac{I_0}{2\pi}
\label{dA/dt},
\ee
from Theorem \ref{Theorem1}, we have
\begin{align*}
\frac{d}{dt} I_{-1}
=
-
\frac{4\pi}{L^2}\frac{dA}{dt}
=
-
\frac{2}{L^2} I_0
\leqq
-
\frac{16\pi^2}{L^2} I_{-1}.
\end{align*}
Therefore, we have the desired conclusion.
\qed
\end{proof}
Next we show the exponential decay of $I_\ell$ for $\ell \in \mathbb{N}$.
\begin{thm}
Let $ \vecf $ be as in Theorem \ref{Theorem10}.
For each $ \ell \in \mathbb{N} \cup \{ 0 \} $,
there exist $ C_\ell > 0 $ and $ \lambda_\ell > 0 $ such that
\[
	I_\ell (t) \leqq C_\ell e^{ - \lambda_\ell t } .
\]
\label{Theorem11}
\end{thm}
\begin{proof}
We initially consider the behavior of $I_0$. By direct calculation, we have
\begin{align*}
\frac{d}{dt} I_0
= & \
\frac{d}{dt} \left(L \int_0^L \kappa^2 ds\right)
\\
= & \
L \int_0^L \left( 2 \nabla_s^2 \veckappa + \| \veckappa \|_{ \mathbb{R}^2 }^2 \veckappa \right)
\cdot
\partial_t \vecf
\, ds
\\
= & \
L \int_0^L 2(\partial_s^2 \tilde{\kappa} + \kappa^3)
\left(\tilde{\kappa} - \frac{1}{2\pi} \int_0^L \tilde{\kappa}^2 ds\right) ds
\\
= & \
-
2L \int_0^L (\partial_s \tilde{\kappa})^2 ds
+
L \int_0^L \kappa^3 \tilde{\kappa} \, ds
-
\frac{1}{2\pi} \int_0^L \kappa^3 ds \int_0^L \tilde{\kappa}^2 ds
\\
= & \
-
\frac{2}{L^2}I_1
+
L \int_0^L 
\left(
\tilde{\kappa}^4 
+ \frac{6\pi}{L} \tilde{\kappa}^3 
+ \frac{12\pi^2}{L^2} \tilde{\kappa}^2
\right) ds
\\
& \quad
- \,
\frac{L}{2\pi} \int_0^L
\left(
\tilde{\kappa}^3
+ \frac{6\pi}{L} \tilde{\kappa}^2
+ \frac{8\pi^3}{L^3}
\right) ds
\int_0^L \tilde{\kappa}^2 ds.
\end{align*}
Hence we obtain
\begin{align}
 \frac{d}{dt} I_0
& + 
\frac{3}{L^2} I_0^2
+
\frac{4\pi^2}{L^2} I_0
+
\frac{2}{L^2}I_1
\label{I_0}\\
& =  
\frac{1}{L^2} \int_0^L 
(L^3 \tilde{\kappa}^4 
+ 6\pi L^2 \tilde{\kappa}^3 
+ 12\pi^2 L \tilde{\kappa}^2
) ds
-
\frac{1}{2\pi L^2}
\left(L^2 \int_0^L \tilde{\kappa}^3 ds \right)
\left(L \int_0^L \tilde{\kappa}^2 ds \right)
\nonumber
.
\end{align}
By the Gagliardo-Nirenberg inequalities, we have
\begin{align*}
\frac{1}{L^2} \int_0^L 
(L^3 \tilde{\kappa}^4 
+ 6\pi L^2 \tilde{\kappa}^3 
+ 12\pi^2 L \tilde{\kappa}^2
) ds
\leqq
	\frac C { L^2 }
	\left(
	I_1^{ \frac 12 } I_0^{ \frac 32 }
	+
	I_1^{ \frac 14 } I_0^{ \frac 54 }
	+
	I_0
	\right).
\end{align*}
Applying Young's inequalities,
Theorem \ref{Theorem1}
and Theorem \ref{Theorem3},
we have
\begin{align*}
	I_1^{ \frac 12 } I_0^{ \frac 32 }
	\leqq & \
	\epsilon I_1 + C_\epsilon I_0^3 ,
	\\
	I_1^{ \frac 14 } I_0^{ \frac 54 }
	\leqq & \
	\epsilon I_1 + C_\epsilon I_0^{ \frac 53 }
	\leqq
	\epsilon ( I_1 + I_0 ) + C_\epsilon I_0^3
	\leqq
	C \epsilon I_1 + C_\epsilon I_0^3
	,
	\\
	I_0
	\leqq & \
	I_{-1}^{ \frac 12 } \left( I_1 + I_1^{ \frac 12 } \right)
	\leqq
	\left( I_{-1}^{ \frac 12 } + \epsilon \right) I_1 + C_\epsilon I_{-1}
	\\
	\leqq & \
	\left( I_{-1}^{ \frac 12 } + \epsilon \right) I_1 + C_\epsilon e^{ - \lambda t }
\end{align*}
for any $ \epsilon > 0 $. Hence the first term on the right-hand side of \pref{I_0} is estimated above by 
\[ 
\left( I_{-1}^{ \frac 12 } + \epsilon \right) I_1 
+ 
\frac{C_\epsilon}{L^2} (I_0^3 + e^{ - \lambda t }).
\]
Moreover we have, by Young's inequality,
\begin{align*}
-
\frac{1}{2\pi L^2}
\left(L^2 \int_0^L \tilde{\kappa}^3 ds \right)
\left(L \int_0^L \tilde{\kappa}^2 ds \right)
& \leqq 
\frac{C}{L^2} I_1^\frac{1}{4} I_0^\frac{5}{4} I_0
=
\frac{C}{L^2} I_1^\frac{1}{4} I_0^\frac{9}{4}
\\
& \leqq 
\frac{\epsilon}{L^2} I_1 + \frac{C_\epsilon}{L^2} I_0^3.
\end{align*}
Taking $\epsilon$ sufficiently small, by Theorem \ref{Theorem10}, we have
\be
\frac{d}{dt} I_0
+
\frac{3}{L^2} I_0^2
+
\frac{4\pi^2}{L^2} I_0
+
\frac{C_1}{L^2} I_1
\leqq
\frac{C_2}{L^2} I_0^3
+
\frac{C_3}{L^2}e^{ - \lambda t }
\label{1}
.
\ee
It is obvious that there exists $ T_1 > 0 $ satisfying
\[
	\int_{ T_1 }^\infty \frac { C_3 } { L^2 } e^{ - \lambda t } dt < \frac 3 { 2 C_2 } .
\]
Furthermore, there exists $ T_2 \geqq T_1 $ such that $ \displaystyle{ I_0 ( T_2 ) < \frac 3 { 2 C_2 } } $,
because we have
\[
	\int_0^\infty I_0 dt \leqq C
\]
by integrating \pref{dA/dt}.
We would like to show $ \displaystyle{ I_0 ( t ) < \frac 3 { C_2 } } $ for $ t \geqq T_2 $.
To do this,
we argue by contradiction.
Then there exists $ T_3 > T_2 $ such that
\[
	I_0 ( t ) < \frac 3 { C_2 } \mbox{ for $ t \in [ T_2 , T_3 ) $,
	and } I_0 ( T_3 ) = \frac 3 { C_2 } .
\]
It follows from \pref{1} that
\[
	\frac d { dt } { I_0 }
	\leqq
	\frac { C_3 } { L^2 } e^{ - \lambda t }
\]
for $ t \in [ T_2 , T_3 ] $.
Hence
\[
	I_0 ( T_3 )
	=
	I_0 ( T_2 )
	+
	\int_{ T_2 }^{ T_3 } \frac d { dt } I_0 dt
	<
	\frac 3 { 2 C_2 }
	+
	\int_{ T_1 }^\infty \frac { C_3 } { L^2 } e^{ - \lambda t } dt
	<
	\frac 3 { C_2 } .
\]
This contradicts $ \displaystyle{ I_0 ( T_3 ) = \frac 3 { C_2 } } $.
Hence we show
\[
I_0 ( t ) < \frac 3 { C_2 } 
\]
for $ t \geqq T_2 $. Therefore, 
from \pref{1}, we obtain
\[
	\frac d { dt } I_0
	+
	\frac {C_1} { L^2 } I_1
	\leqq
	\frac{C_3}{L^2}e^{ - \lambda t }.
\]
By Wirtinger's inequalities, we have
\[
	\frac d { dt } I_0
	+
	\frac {C_4} { L^2 } I_0
	\leqq
	\frac{C_3}{L^2}e^{ - \lambda t }.
\]
Thus the assertion for $ \ell = 0 $ with some positive $ \lambda_0 $ has been proved.
\par
Next we consider the behavior of $I_\ell$ for $\ell \in \mathbb{N}$. 
Set
\[
	J_{k,p}
	=
	\left\{ L^{ ( 1 + k ) p - 1 } \int_0^L | \partial_s^k \tilde \kappa |^p ds \right\}^{ \frac 1p }
	.
\]
By the Gagliardo-Nirenberg inequalities we have
\be
	J_{ k,p } \leqq C J_{m,2}^\theta J_{0,2}^{ 1 - \theta }
	=
	C I_m^{ \frac \theta 2 } I_0^{ \frac { 1 - \theta } 2 }
	\label{J}
\ee
for $ k \in \{ 0 , 1 , \ldots , m \} $,
$ p \geqq 2 $.
Here $ C $ is independent of $ L $,
and $ \theta =  \frac 1m  \left( k - \frac 1p + \frac 12 \right)  \in [ 0,1 ] $.
For $ k \in \mathbb{N} \cup \{ 0 \} $ and $ m \in \mathbb{N} $,
let $ P_m^k ( \tilde \kappa ) $ be any linear combination of the type
\[
	P_m^k ( \tilde \kappa )
	=
	\sum_{ i_1 + \cdots + i_m = k } c_{ i_1 , \dots , i_m }
	\partial_s^{ i_1 } \tilde \kappa \cdots \partial_s^{ i_m } \tilde \kappa
\]
with universal,
constant coefficients $ c_{ i_1 , \dots , i_m } $.
Similarly we define $ P_0^k $ as a universal constant.
We can show
\[
	\partial_t \partial_s^k \tilde \kappa
	=
	\partial_s^{ k+2 } \tilde \kappa
	+
	\sum_{ m=0 }^2 L^{ - ( 2 - m ) } P_{m+1}^k ( \tilde \kappa )
	+
	\sum_{ m=0 }^1 L^{ - ( 1 - m ) } P_{m+1}^k ( \tilde \kappa ) \int_0^L \tilde{\kappa}^2 ds
\]
by induction on $k$.
Hence we have
\begin{align*}
\frac{d}{dt} I_\ell
= & \
2L^{2\ell+1} \int_0^L \partial_s^\ell \tilde{\kappa} \partial_t \partial_s^\ell \tilde \kappa \, ds
+
L^{2\ell+1} \int_0^L (\partial_s^\ell \tilde{\kappa})^2 \partial_t(ds)
\\
= & \
-
2L^{2\ell+1} \int_0^L (\partial_s^{\ell+1} \tilde \kappa)^2 ds
+
2L^{2\ell+1} \int_0^L \partial_s^\ell \tilde \kappa \sum_{ m=0 }^2 L^{ - ( 2 - m ) } P_{m+1}^\ell ( \tilde \kappa ) ds
\\
& \quad
+ \,
2L^{2\ell+1} \int_0^L \partial_s^\ell \tilde \kappa \sum_{ m=0 }^1 L^{ - ( 1 - m ) } P_{m+1}^\ell ( \tilde \kappa ) ds
\int_0^L \tilde{\kappa}^2 ds
-
L^{2\ell+1} \int_0^L (\partial_s^\ell \tilde \kappa)^2 \tilde{\kappa} \kappa \, ds
\\
& \quad
+ \,
\frac{L^{2\ell+1}}{2\pi} \int_0^L (\partial_s^\ell \tilde \kappa)^2 \kappa \, ds
\int_0^L \tilde{\kappa}^2 ds
\\
= & \
-
2L^{2\ell+1} \int_0^L (\partial_s^{\ell+1} \tilde \kappa)^2 ds
+
2\sum_{ m=0 }^2 L^{2\ell+1} \int_0^L \partial_s^\ell \tilde \kappa  L^{ - ( 2 - m ) } P_{m+1}^\ell ( \tilde \kappa ) ds
\\
& \quad
+ \,
2\sum_{ m=0 }^1 L^{2\ell+1} \int_0^L \partial_s^\ell \tilde \kappa  L^{ - ( 2 - m ) } P_{m+1}^\ell ( \tilde \kappa ) ds
\int_0^L \tilde{\kappa}^2 ds.
\end{align*}
We define $\Phi_m$ to be
\[
\Phi_m
=
L^{2\ell+1} \int_0^L \partial_s^\ell \tilde \kappa  L^{ - ( 2 - m ) } P_{m+1}^\ell ( \tilde \kappa ) ds.
\]
When $ m = 0 $, we have
\[
\Phi_0
=
\frac{C}{L^2} I_\ell
\]
since $P_1^\ell(\tilde{\kappa}) = c\partial_s^\ell\tilde{\kappa}$.
When $ m = 1 $, we have
\[
\Phi_1
=
L^{2\ell+1} \int_0^L \partial_s^\ell \tilde \kappa  L^{-1} P_{2}^\ell ( \tilde \kappa ) ds.
\]
Also $ P_2^\ell ( \tilde \kappa ) $ is a linear combination of $ ( \partial_s^k \tilde \kappa ) ( \partial_s^{ \ell - k } \tilde \kappa ) $ with $ k = 0 $,
$ \cdots $,
$ \ell $.
By H\"{o}lder's inequality,
we have
\[
	\left| L^{ 2 \ell + 1 } \int_0^L
	( \partial_s^\ell \tilde \kappa ) L^{-1} P_2^\ell ( \tilde \kappa ) \, ds
	\right|
	\leqq
%	C L^{ 2 \ell }
%	\| \partial_s^\ell \tilde \kappa \|_{ L^3 }
%	\| \partial_s^k \tilde \kappa \|_{ L^3 }
%	\| \partial_s^{ \ell - k } \tilde \kappa \|_{ L^3 }
%	=
	\sum_{k=0}^\ell \frac C { L^2 } J_{ \ell , 3 } J_{ k,3 } J_{ \ell - k , 3 }
	,
\]
and \pref{J} yields
\[
	J_{ j , 3 } \leqq C I_{ \ell + 1 }^{ \frac { \theta (j,3) } 2 } I_0^{ \frac { 1 - \theta (j,3) } 2 }
	,
	\quad
	\theta ( j,3 ) = \frac { j + \frac 16 } { \ell + 1 } .
\]
Hence applying Young's inequality,
we obtain
\[
	\left| L^{ 2 \ell + 1 } \int_0^L
	( \partial_s^\ell \tilde \kappa ) L^{-1} P_2^\ell ( \tilde \kappa ) \, ds
	\right|
	\leqq
	\frac C { L^2 } I_{ \ell + 1 }^{ \frac { 2 \ell + \frac 12 } { 2 ( \ell + 1 ) } }
	I_0^{ \frac { \ell + \frac 52 } { 2 ( \ell + 1 ) } }
	\leqq
	\frac \epsilon { L^2 } I_{ \ell + 1 }
	+
	\frac { C_\epsilon } { L^2 } I_0^{ \frac { 2 \ell + 5 } 3 }
\]
for any $ \epsilon > 0$.
When $ m = 2 $, we obtain
\begin{align*}
\Phi_2
= & \
L^{2\ell+1} \int_0^L \partial_s^\ell \tilde \kappa P_{3}^\ell ( \tilde \kappa ) ds.
\end{align*}
We can estimate for the case $ m = 3 $ similarly.
Indeed,
since $ P_3^\ell ( \tilde \kappa ) $ is a linear combination of $ \partial_s^j \tilde \kappa $,
$ \partial_s^k \tilde \kappa $,
and $ \partial_s^{ \ell - j - k } \tilde \kappa $,
we have
\begin{align*}
	\left| L^{ 2 \ell + 1 } \int_0^L
	( \partial_s^\ell \tilde \kappa ) P_3^\ell ( \tilde \kappa ) \, ds
	\right|
	&\leqq
%	C L^{ 2 \ell + 1 }
%	\| \partial_s^\ell \tilde \kappa \|_{ L^4 }
%	\| \partial_s^j \tilde \kappa \|_{ L^4 }
%	\| \partial_s^k \tilde \kappa \|_{ L^4 }
%	\| \partial_s^{ \ell - j - k } \tilde \kappa \|_{ L^4 }
%	=
	\sum_{m=0}^\ell \sum_{\substack{j+k=m\\ j\geqq 0,k\geqq 0}} \frac C { L^2 } J_{ \ell , 4 } J_{ j,4 } J_{ k,4 } J_{ \ell - j - k , 4 }\\
	&\leqq
	\frac C { L^2 } I_{ \ell + 1 }^{ \frac { 2 \ell + 1 } { 2 ( \ell + 1 ) } } I_0^{ \frac { 2 \ell + 3 } { 2 ( \ell + 1 ) } }
	\leqq
	\frac \epsilon { L^2 } I_{ \ell + 1 }
	+
	\frac { C_\epsilon } { L^2 } I_0^{ 2 \ell + 3 }
\end{align*}
for any $ \epsilon > 0$.
Hence we obtain
\begin{align*}
\frac{d}{dt} I_\ell
+
\frac{C_1}{L^2} I_{\ell+1}
\leqq & \
\frac{C_2}{L^2} I_\ell
+
\frac{\epsilon}{L^2} I_{\ell+1}
+
\frac{C_\epsilon}{L^2} I_0^\frac{2\ell +5}{3}
+
\frac{C_\epsilon}{L^2} I_0^{2\ell+3}
+
\frac{1}{L^2} I_\ell I_0
\\
& \quad
+ \,
\frac{\epsilon}{L^2} I_{\ell+1} I_0 
+
\frac{C_\epsilon}{L^2} I_0^{2\ell+4}. 
\end{align*}
From Theorem \ref{Theorem3} and Young's inequality, we have
\[
\frac{C}{L^2} I_\ell
\leqq 
\frac{C}{L^2}
\left(
I_{-1}^\frac{1}{2} I_{\ell+1}
+
I_{-1}^\frac{1}{\ell+2} I_{\ell+1}^\frac{\ell+1}{\ell+2}
\right)
\leqq 
\frac{C}{L^2}
\left\{
( I_{-1}^\frac{1}{2} + \epsilon ) I_{\ell+1}
+
C_\epsilon I_{-1}
\right\}.
\]
Therefore we have
\[
\frac{d}{dt} I_\ell
+
\frac{C_1}{L^2} I_\ell
\leqq
\frac{C}{L^2}
\left(
I_0^\frac{2\ell +5}{3}
+
I_0^{2\ell+3}
+
I_0^{2\ell+4}
+
I_{-1}
\right).
\]
Since we have already shown that $I_{-1}$ and $I_0$ decay exponentially as $t \to \infty$, we get the desired conclusion for 
$\ell \in \mathbb{N}$.
\qed
\end{proof}
\par
We can prove the following theorem in a similar way to the proof of Theorem \ref{Theorem6}, where we use Theorem \ref{Theorem10} instead of Theorem \ref{Theorem4}.
\begin{thm}
The claims (1)--(7) in Theorem \ref{Theorem6} also hold for global solutions of the length-preserving flow.
\label{Theorem12}
\end{thm}

\par\noindent
{\bf Acknowledgment}.
The author expresses their appreciation to Professor Shigetoshi Yazaki and Professor Tetsuya Ishiwata for sharing information of related articles,
and for discussions.
The author also would like to express their gratitude to Professor Neal Bez for English language editing.

\end{document}